\documentclass{article}

\usepackage[utf8]{inputenc}
\usepackage{authblk}
\usepackage{setspace}
\usepackage[margin=1.25in]{geometry}
\usepackage{graphicx}
\graphicspath{ {./figures/} }
\usepackage{subcaption}
\usepackage{amsmath}
\usepackage{makecell}
\usepackage{lineno}
\usepackage{amsfonts}
\usepackage{amssymb, amsthm, amsmath, amsfonts}
\usepackage{wasysym}
\usepackage{mathrsfs}
\usepackage{hyperref}
\usepackage{graphicx}
\usepackage{lineno}
\usepackage[colorinlistoftodos]{todonotes}
\usepackage{listings}
\usepackage{cancel, enumerate}
\usepackage{rotating, environ}
\usepackage{caption}
\usepackage{subcaption}
\usepackage[inline]{enumitem}
\usepackage{dirtree}
\usepackage{xcolor}
\usepackage{tipa}
\usepackage{float}
\usepackage{changepage}

\def\sbng{\bngviii}

\def\bng{\bngx}
\def\lbng{\bngxiv}

\font\bngviii=bang10 scaled 800

\font\bngx=bang10

\font\bngxiv=bang10 scaled 1400

\def\*#1*#2{o\null{#2}{#1}}

\def\sh#1{\setbox0=\hbox{#1}%
     \kern-.02em\copy0\kern-\wd0
     \kern.04em\copy0\kern-\wd0
     \kern-.02em\raise.0433em\box0 }

\newcommand{\haw}{\text{\bng H}}
\newcommand{\footnoteh}{\text{\sbng H}}
\newcommand{\sectionh}{\text{\lbng H}}

\newtheorem{defn}{Definition}
\newtheorem{prop}{Proposition}
\newtheorem{lemma}{Lemma}
\newtheorem{cor}{Corollary}
\newtheorem{ex}{Example}
\newtheorem{conj}{Conjecture}
\newtheorem{rmk}{Remark}
\newtheorem{qstn}{Question}
\newtheorem{thm}{Theorem}

\renewcommand{\P}{\mathbb{P}}
\newcommand{\subseq}{\subseteq}
\newcommand{\N}{\mathbb{N}}
\newcommand{\Z}{\mathbb{Z}}

\newcommand{\Q}{\mathbb{Q}}

\newcommand{\R}{\mathbb{R}}

\newcommand{\ds}{\displaystyle}

\newcommand{\Aut}{\text{Aut}}

\newcommand{\Id}{\text{Id}}
\newcommand{\C}{\mathbb{C}}

\newcommand{\F}{\mathbb{F}}

\newcommand{\1}{\mathds{1}}

\renewcommand{\P}{\mathbb{P}}

\newcommand{\inv}{^{-1}}

\newcommand{\dv}{\mathsf{d}}

\makeatletter
\newcommand{\tpitchfork}{%
  \vbox{
    \baselineskip\z@skip
    \lineskip-.52ex
    \lineskiplimit\maxdimen
    \m@th
    \ialign{##\crcr\hidewidth\smash{$-$}\hidewidth\crcr$\pitchfork$\crcr}
  }%
}
\makeatother

\newcommand{\Pcal}{\mathcal{P}}

\newcommand{\+}{\oplus}
\renewcommand{\H}{\mathcal{H}}
\newcommand{\hC}{\widehat{\C}}

\newcommand{\ddeg}{\text{ddeg\,}}
\newcommand{\Pk}{\mathfrak{P}}

\title{The dynamics of the Hesse derivative on the $j$-invariant}

\author{Jake Kettinger}

\affil{Colorado State University}

\date{25 March 2026}

\begin{document}

\maketitle

\begin{abstract}
In this paper, we study the Hesse derivative of a cubic curve on the set of $j$-invariants, which can be viewed as a rational function on the Riemann sphere.  We then analyze the dynamics of this rational function, including counting the number of orbits of a given size. We proceed to investigate when a cubic curve is isomorphic to its $n$-fold Hesse derivative, showing that when an elliptic curve has a $j$-invariant that is periodic under this rational function, the curve itself must be periodic under the Hesse derivative. We finish with some data collected comparing the sizes of the orbits of elliptic curves to those of their $j$-invariants, and some further questions about this dynamical system.
\end{abstract}

\section{Introduction}
In this paper we are interested in studying the dynamics of the rational function on the Riemann sphere defined by \begin{align*}H(j)=\frac{(6912-j)^3}{27j^2}.\end{align*} In Theorem 3.4 of \cite{PP}, this function was shown to relate the $j$-invariant of an elliptic curve to the $j$-invariant of its Hesse derivative. The dynamics of this function were subsequently explored in \cite{CS}, and in any characteristic different from 2 or 3 (0 included) \cite{MPT}. In Section 2 of this paper, we provide a direct proof of the derivation of this function $H$ in Proposition \ref{jdyn}. In the subsequent section, we delve into the dynamical properties of this function, culminating in Theorem \ref{thm2}, wherein we compute the number of orbits of any given size. This is a similar computation to that of the number of Hesse loops of length of even length provided in Theorem 22 of \cite{DHH}.  In the fourth section of this paper, we study the effect of the Hesse derivative on the fibers of an elliptic fibration. We use the Hessian group of an elliptic fibration to prove Theorem \ref{ffo}, which states that an elliptic curve whose $j$-invariant is periodic under $H$ is itself periodic under the Hesse derivative. Interestingly, the size of the orbit of $j$ under $H$ does not uniquely determine the size of the orbit of the curve under the Hesse derivative, as the data tables at the end of the section show. We end the paper with some open questions for further pursuit. 


\section{Background}
\subsection{The Hesse Derivative}
Throughout this paper we will work in the complex projective plane $\P_\C^2$, which we will henceforth denote $\P^2$. Let $f$ be a reduced homogenous polynomial of degree 3 in 3 variables. Then $f$ defines a cubic curve $V(f)$ in $\P^2$. It is known by Prop 3.1, Ch III of \cite{S} that $V(f)$ is isomorphic to $V(w)$ where $$w=y^2z-x^3-axz^2-bz^3$$ where $a,b\in\C$. This expression is called the \textbf{Weierstrass form} of the curve. The $j$-invariant of a cubic curve can be calculated from an isomorphic Weierstrass form as follows: $$j=1728\cdot\frac{4a^3}{4a^3+27b^2}\in\hC,$$ where $\hC$ is the Riemann sphere.

In this paper we will be investigating the effect of the Hesse derivative on the $j$-invariant of a cubic curve. The Hesse derivative $\haw C$\footnote{We use the notation introduced in \cite{DHH} of using the Bengali letter $\footnoteh$ (pronounced ``Haw") to denote the Hesse derivative.} of a curve $C$ given by the polynomial $f$ is the curve determined by the determinant of the Hesse matrix $$\H(f)=\left|\begin{matrix}
    f_{xx}&f_{xy}&f_{xz}\\
    f_{xy}&f_{yy}&f_{yz}\\
    f_{xz}&f_{yz}&f_{zz}\\
\end{matrix}\right|.$$ As one can see the degree of the Hesse derivative of a curve of degree $d$ is $3(d-2)$ if the Hessian does not vanish. Thus when $\deg C=3$, we have $\deg\haw C=3$. So the Hesse derivative can be viewed as a self-map on the family of cubic curves in $\P^2$. When $C$ is a trilateral, $\haw C=C$.

\subsection{Computing the $j$-invariant of $\sectionh C$} The goal of this section is to compute the $j$-invariant $j(\haw C)$ of $\haw C$ in terms of $j(C)$. We will begin with the following lemma.

\begin{lemma}\label{jlemma}
Let $C$ be an elliptic curve given by the polynomial in Weierstrass form $$f=y^2z-x^3-axz^2-bz^3.$$ Then \[j(\haw C)=\begin{cases}j(C)\cdot\ds\frac{(a^3+9b^2)^3}{a^9}&a\neq 0,\\
\infty&a=0.\\\end{cases}\]
\end{lemma}
\begin{proof}

The $a=0$ case comes from the fact that then $\haw C$ is a trilateral, so $j(\haw C)=\infty$. Assuming $a\neq 0$, applying the Hessian to $f$ gives us the polynomial $$\H(f)(x,y,z)=24xy^2-8a^2z^3+24ax^2z+72bxz^2=8( 3xy^2-a^2z^3+3ax^2z+9bxz^2).$$ Our goal now is to perform a sequence of linear transformations on this polynomial to achieve a Weierstrass form, so that we may compute the $j$-invariant.  By performing a change of coordinates, we can attain a cubic curve isomorphic to $\H(f)$ that is in Weierstrass form; specifically, by defining $\alpha:=3a^2$, $\beta:=-a^8$, $\gamma:=3a^7+27a^4b^2$, and $\delta:=9a^5b+54a^2b^3$, we have $$\H(f)\left(a^2z,\sqrt{\beta}y,-a^2\sqrt[3]{\alpha}x+3bz\right)/(8\alpha\beta)=y^2z-x^3-\frac{\sqrt[3]{\alpha}\gamma}{\alpha\beta}xz^2-\frac{-\delta}{\alpha\beta}z^3$$ in Weierstrass form.
%

Now recall that the $j$-invariant of $C$ is $j(C)=1728\cdot\ds\frac{4a^3}{4a^3+27b^2}$.

Since the curves defined by the polynomials $\H(f)(x,y,z)$ and $\frac{\H(f)\left(a^2z,\sqrt{\beta}y,-a^2\sqrt[3]{\alpha}x+3bz\right)}{8\alpha\beta}$ are isomorphic, they have the same $j$-invariant, and does not depend of the choice of roots $\sqrt[3]{\alpha},\sqrt{\beta}$.




Mathematica  \cite{Mathematica} tells us \begin{align*}\H(f)\left(a^2z,\sqrt{\beta}y,-a^2\sqrt[3]{\alpha}x+3bz\right)/(8\alpha\beta)\\=-x z^2 \left(-\frac{9 \sqrt[3]{3} \sqrt[3]{a^2} b^2}{a^6}-\frac{\sqrt[3]{3} \sqrt[3]{a^2}}{a^3}\right)-z^3 \left(\frac{18 b^3}{a^8}+\frac{3 b}{a^5}\right)-x^3+y^2 z.\end{align*}

Using Mathematica, the $j$-invariant of this is $$j(\haw C)=1728\cdot\frac{4(a^3+9b^2)^3}{a^6(4a^3+27b^2)}=j(C)\cdot\frac{(a^3+9b^2)^3}{a^9}.$$

\end{proof}
The \href{https://arxiv.org/abs/2408.04117v2}{earlier version of this article on arXiv} also has some Macaulay2 code available to help further elucidate this proof if so desired.

Now we can improve the result of the lemma and write $j(\haw C)$ purely in terms of $j(C)$.

\begin{prop}\label{jdyn}
    Consider an elliptic curve $C$ with $j$-invariant $j$. Then the $j$-invariant of $\haw C$ is $\ds\frac{(6912-j)^3}{27j^2}$. This coincides with Theorem 3.4 of \cite{PP}.
\end{prop}
\begin{proof}
By Lemma \mbox{\ref{jlemma}}, $j(\haw C)=j(C)\cdot\ds\frac{(a^3+9b^2)^3}{a^9}=j(C) \left(\ds\frac{a^3+9b^2}{a^3}\right)^3=\displaystyle j(C)\left(\frac{1728}{j(C)}+\frac{9b^2}{4a^3}\right)^3.$ 

Now note that $$\ds\frac{9b^2}{4a^3}-\frac{1728}{j(C)}=\frac{9b^2}{4a^3}-\frac{4a^3+27b^2}{4a^3}=\frac{-4a^3-18b^2}{4a^3}=-1-2\left(\frac{9b^2}{4a^3}\right).$$ Solving for $\ds\frac{9b^2}{4a^3}$, we get $\ds\frac{9b^2}{4a^3}=\frac{1728j(C)\inv-1}{3}$.

Therefore $j(\haw C)=j(C)\cdot\left(\ds\frac{1728}{j(C)}+\frac{1728j(C)\inv-1}{3}\right)^3$. Simplifying this yields $$j(\haw C)=\frac{(6912-j(C))^3}{27j(C)^2}.$$
\end{proof}
It is worth observing that the two most stand-out fixed points of this rational function are $j=1728$ and $j=\infty$. The latter $j$-invariant corresponds to that of a trilateral $T$, in which case $\haw T=T$. The subsequent sections will conduct a more thorough survey of the periodic $j$-invariants.
\section{The dynamics of $H$}
\subsection{Periodic $j$-invariants}
For an elliptic curve $C$, we will use the rational function \begin{align*}&H:\widehat{\C}\to \widehat{\C}\\ \text{defined by }& H(j)=\ds\frac{(6912-j)^3}{27j^2}\end{align*} 
to understand the behavior of $j(\haw^nC)$.

\begin{defn}
Let $f:\hC\to\hC$ be a rational function and $p\in\hC$. Then the \textbf{forward orbit} (or simply \textbf{orbit}) of $p$ is $\text{Orb}_f(p)=\{f^{n}(p):n>0\}$,where $f^{ n}$ denotes the $n$-fold composition of $f$ with itself.
\end{defn}
\begin{defn}
Let $f:\hC\to\hC$ be a rational function. Then $p\in\hC$ is \textbf{fixed} if $f(p)=p$. The point $p$ is \textbf{periodic} if $f^{n}(p)=p$ for some $n>0$. The point $p$ is called \textbf{preperiodic} if there exists an $n>0$ where $f^{n}(p)$ is periodic. In the case $p$ is periodic, its orbit $\text{Orb}_f(p)$ is a \textbf{cycle}.

\end{defn}

\begin{ex}
Note that $H(j)=1728$ whenever $(6912-j)^3=46656j^2$. This gives us the polynomial $j^3+25920j^2+3\cdot 6912^2j-6912^3=(j-1728)(j+13824)^2$. This gives us two roots: one at $j=1728$ and one at $j=-13824=-8\cdot 1728$, which are fixed and preperiodic, respectively.
\end{ex}
\begin{ex}\label{order2}
Note that $H(H(j))$ can be rewritten as $$H(H(j))=\frac{\left(j^3+165888 j^2+143327232 j-330225942528\right)^3}{729 (j-6912)^6 j^2}.$$ Thus $H(H(j))=j$ when $$\left(j^3+165888 j^2+143327232 j-330225942528\right)^3=729 (j-6912)^6 j^3.$$ Along with the fixed point of $j=\infty$, this gives us a degree-9 polynomial with roots \begin{align*}
    j&=1728\\
    j&=\frac{3456}{7} \left(-1-3 i \sqrt{3}\right)\\
    j&=\frac{3456}{7} \left(-1+3 i \sqrt{3}\right)\\
    j&=3456 \left(5-3 \sqrt{3}\right)\\
    j&=3456 \left(5+3 \sqrt{3}\right)\\
    j&=-i5184  \sqrt{3}-\frac{1}{2} \sqrt{-\frac{4514807808}{13}-i\frac{644972544 \sqrt{3}}{13}}+1728\\
    j&=-i5184  \sqrt{3}+\frac{1}{2} \sqrt{-\frac{4514807808}{13}-i\frac{644972544  \sqrt{3}}{13}}+1728\\
    j&=i5184  \sqrt{3}-\frac{1}{2} \sqrt{-\frac{4514807808}{13}+i\frac{644972544 \sqrt{3}}{13}}+1728\\
    j&=i5184 \sqrt{3}+\frac{1}{2} \sqrt{-\frac{4514807808}{13}+i\frac{644972544  \sqrt{3}}{13}}+1728.\\
\end{align*}
So when the $j$-invariant of a cubic curve $C$ is one of these ten points of $\hC$, then $\haw^2C\cong C$. Note that for the latter six of these values, we will have $\haw C\not\cong C$. For example, $H$ transposes the two values $3456\left(5-3\sqrt{3}\right)$ and $3456\left(5+3\sqrt{3}\right)$.
\end{ex}
An interesting question that we investigate in the next section is when do we have $\haw^n C= C$ instead of simply $\haw^n C\cong C$. It is already known that $\haw C=C$ when $j(C)=\infty$ (when $C$ is a trilateral), and that $\haw^2 C=C$ when $j(C)\in\{\infty,1728\}$. In the $j(C)=1728$ case, $C$ is called harmonic \cite{CO}.

\begin{defn}
Let $f:\hC\to\hC$ be a rational function, and let $\{p_1,p_2,\dots,p_m\}\subseq\C$ form a cycle of size $m$ under $f$. That is, the $p_i$ are periodic points of period $m$. Then the \textbf{multiplier} $\lambda$ of the orbit is $$(f^{m})'(p_i)=\prod_{i=1}^mf'(p_i),$$ see \cite{B}. In the special case that one of the $p_i$ is $\infty$, then $$\lambda=\frac{1}{(f^{m})'(\infty)}.$$
\end{defn}

\begin{defn}
For a rational function $f:\hC\to\hC$, a cycle with multiplier $\lambda$ is called \textbf{superattracting}, \textbf{attracting}, \textbf{repelling}, or \textbf{indifferent} if $\lambda=0$, $|\lambda|<1$, $|\lambda|>1$, or $|\lambda|=1$, respectively. Furthermore, the cycle is called \textbf{rationally indifferent} if $\lambda$ is a root of unity, and \textbf{irrationally indifferent} if $|\lambda|=1$ but $\lambda$ is not a root of unity.
\end{defn}

We will begin by finding the Julia set of $H$, which is defined as follows.

\begin{defn}
For a rational function $f:\hC\to\hC$ the \textbf{Julia set} $J(f)$ is the closure of the set of repelling periodic points. 
\end{defn}

\begin{defn}
Let $f:\hC\to\hC$ be a rational function. Then $f$ is \textbf{post-critically finite} (\textbf{PCF}) if all the critical points (i.e. points where the derivative is 0 or undefined) of $f$ are preperiodic.
\end{defn}

It is shown in Corollary 16.5 of \cite{M} that a PCF rational function with no periodic critical points has the entire Riemann sphere $\hC$ as its Julia set.

\begin{prop}
The rational function $H$ is a post-critically finite function with no periodic critical points, so $J(H)=\hC$.
\end{prop}
\begin{proof}
The derivative of $H$ is $$H'(j)=-\frac{(6912-j)^2(j+13824)}{27j^3}.$$ Thus $H$ has three critical points: $6912$, $-13824$, and $0$. Note that $\infty$ is not a critical point because $H'(\infty)=-\frac{1}{27}$. We will see that all three critical points are preperiodic.

First note that $H(6912)=0$, $H(0)=\infty$, and $H(\infty)=\infty$ is a fixed point. Thus $j=6912$ is a preperiodic point, but not periodic.

Next, we have $H(-13824)=1728$, and $H(1728)=1728$ is a fixed point, so $j=-13824$ is a preperiodic point, but not periodic.

Finally, we already saw that $H(0)=H(H(0))=\infty$, so $j=0$ is a preperiodic point, but not periodic.

Thus all three critical points of $H$ are preperiodic, and so $H$ is PCF. Since none of the critical points are periodic, we then know that $J(H)=\hC$.
\end{proof}




As a map $H:\widehat{\C}\to\widehat{\C}$ on the Riemann sphere, $H$ has four fixed points: $$1728,\,\ds\frac{3456}{7}(-1+3i\sqrt{3}),\,\ds\frac{3456}{7}(-1-3i\sqrt{3}),\text{ and }\infty.$$ The four fixed points have $\lambda$ multipliers of $$-3,\,\ds-\frac{3}{2}-i\frac{\sqrt{3}}{2},\,\ds-\frac{3}{2}+i\frac{\sqrt{3}}{2},\text{ and }-27$$ respectively. All of the multipliers are have absolute values greater than 1, so all four fixed points are repelling.

\subsection{The number of orbits of size $n$}

\begin{defn}
    Let $\ds\frac{f(z)}{g(z)}:\hC\to\hC$ be a rational function, with $f,g$ polynomials. Then the \textbf{degree} of $f(z)/g(z)$ is $\deg(f/g)=\max\{\deg f(z),\deg g(z)\}$. The \textbf{difference-degree} of $f(z)/g(z)$ is $\emph{ddeg\,}(f/g)=\deg(f)-\deg(g)$.
\end{defn}
\begin{defn}
Let $\ds\frac{f(z)}{g(z)}:\hC\to\hC$ be a rational function, with $f,g$ polynomials. Then $f(z)/g(z)$ is \textbf{reduced} if $f$ and $g$ have no roots in common.
\end{defn}
Going forward, we will only consider rational functions with integer coefficients.
\begin{lemma}\label{redcomp}
Let $f(z)=p(z)/q(z)$ be a reduced difference-degree-1 rational function where $n=\deg p=\deg q+1$ and let $g(z)=u(z)/v(z)$ be a reduced rational function.
    Then $$\frac{p(u/v)v^n}{q(u/v)v^n}$$ is a reduced expression for $f(g(z))$ as the quotient of two polynomials.
\end{lemma}
\begin{proof}
    Write $p(z)=a_0+a_1z+\cdots a_nz^n$ and write $q(z)=b_0+b_1z+\cdots+b_{n-1}z^{n-1}$. Then \begin{align*}
        P:=p(u/v)v^n&=a_0v^n+a_1uv^{n-1}+\cdots+a_{n-1}u^{n-1}v+a_nu^n,\\
        Q:=q(u/v)v^n&=b_0v^n+b_1uv^{n-1}+\cdots+b_{n-1}u^{n-1}v.
    \end{align*}
    We wish to show that the two polynomials $P$ and $Q$ have no roots in common.

    First suppose $r$ is a root of $P$. Then $u(r)/v(r)$ is a root of $p$. We know this by first seeing that if $r$ is a root of $P$ then $r$ cannot be a root of $v$, for if $P(r)=v(r)=0$ then $a_nu(r)^n=0$ and so $r$ is a root of $u$, which is a contradiction. Thus $p(u/v)(r)=0$ and so $u(r)/v(r)$ is a root of $p$.
    
    Now suppose $r$ is a root of $Q$. Then either $u(r)/v(r)$ is a root of $q$ or $r$ is a root of $v$. Thus we can break the property of $r$ simultaneously being a root of $P$ and $Q$ into two cases.

    Case 1: Suppose $u(r)/v(r)$ is a root of $p$ and $q$. This is a contradiction as $p/q$ was assumed to be reduced.

    Case 2: Suppose $u(r)/v(r)$ is a root of $p$ and $r$ is a root of $v$. Then $u/v$ has a pole at $z=r$ because $r$ cannot be a root of $u$. Then $(u/v)(r)=\infty\in\widehat{\C}$, which cannot be a root of $p$ since $p$ is a polynomial, so $p(\infty)=\infty\neq 0$. Thus $P$ and $Q$ cannot have a root in common.
\end{proof}
\begin{rmk}\label{leadingcoefficients}
    In the case $g(z)$ has difference-degree 1, let $u(z)=c_0+c_1z+\cdots c_nz^m$ and $v(z)=d_0+d_1z+\cdots d_{m-1}z^{m-1}$. Then the leading (integer) coefficients of $P$ and $Q$ in the variable $z$ are $a_nc_m^n$ and $b_{n-1}d_{m-1}c_m^{n-1}$, respectively. Thus they are only equal if $a_nc_m=b_{n-1}d_{m-1}$. Specifically, in the case $f(z)=g(z)=H(z)=\ds\frac{(6912-z)^3}{27z^2}$, we have a leading coefficient for $P$ of $1$ and a leading coefficient for $Q$ of $27^2$. In general, the leading coefficient of the numerator of $H^n(z)$ is $(-1)^n$ and that of the denominator is $27^n$.
\end{rmk}

\begin{prop}
Let $f(z)=\ds\frac{p(z)}{q(z)}$ and $g(z)=\ds\frac{u(z)}{v(z)}$ be difference-degree 1 rational functions. Then $\emph{ddeg\,} f(g(z))=1$. Furthermore, $\deg f(g(z))=\deg(f)\deg(g)$.
\end{prop}
\begin{proof}
First note that $\deg(f(g(z))=\deg(f)\deg(g)$ is proven in \cite{B}, pp. 32. By Exercise 2.5 of \cite{B}, the difference-degree of $f(g(z))$ is equal to the \textit{valency} of $f(g(z))$ at $\infty$: that is, the number of solutions to $f(g(z))=f(g(\infty))$ at $\infty$. By the chain rule, the valency of $f(g(z))$ at $\infty$ is 1 \cite{B}, and so $\ddeg f(g)=1$.
\end{proof}
%
%
%
\begin{cor}\label{Hcomp}
    Let $H:\hC\to\hC$ satisfy $H(z)=\ds\frac{(6912-z)^3}{27z^2}$. Let $H^n$ be the $n$-fold composition of $H$ with itself. Then $\deg H^n=3^n$ and $\emph{ddeg\,}H^n=1$.
\end{cor}

\begin{prop}\label{mult1}
    Let $F:\hC\to\hC$ be a PCF rational function and $\phi$ be a fixed point of $F^n(z)=P_n(z)/Q_n(z)$, where $P_n$ and $Q_n$ are coprime polynomials. Furthermore let $n$ be the smallest positive integer $i$ for which $F^i(\phi)=\phi$. Then the multiplicity of $\phi$ as a root of $P_n(z)-zQ_n(z)$ is $1$.
\end{prop}
\begin{proof}
    First we will demonstrate that if the multiplicity of $\phi$ as a root of $P_n(z)-zQ_n(z)$ is greater than 1, then $\left.\ds\frac{\dv}{\dv z}F^n(z)\right|_{z=\phi}=1$. This is simply because if $F^n(\phi)-\phi=0$ and $\left.\ds\frac{\dv}{\dv z}[F^n(z)-z]\right|_{z=\phi}=0$ (by definition of multiplicity), then $\left.\ds\frac{\dv}{\dv z}F^n(z)\right|_{z=\phi}-\left.\ds\frac{\dv}{\dv z}z\right|_{z=\phi}=0$, so $\left.\ds\frac{\dv}{\dv z}F^n(z)\right|_{z=\phi}=1$.

    Now note that $\ds\frac{\dv}{\dv z}F^n(z)=\ds\prod_{i=0}^{n-1}F'(F^i(z))$ by the chain rule. So in the case of the fixed point $\phi$, $$\left.\frac{\dv}{\dv z}F^n(z)\right|_{z=\phi}=\prod_{i=0}^{n-1}F'(F^i(\phi)).$$ Because $n$ is the smallest integer for which $F^n(\phi)=\phi$, the set $\{F^i(\phi)\}_{i=0}^{n-1}$ is the orbit of $\phi$. Therefore the product above is the multiplier for the cycle of $\phi$. Now still supposing the multiplicity of $\phi$ as a root of $P_n(z)-zQ_n(z)$ is greater than 1, this places $\phi$ in a rationally indifferent cycle of $F$.
    
    
   This contradicts Corollary 14.5 of \cite{M}, which says that since $F$ is PCF, every cycle of $F$ is either superattracting or repelling; there are no rationally indifferent cycles of $F$. Hence the multiplicity of $\phi$ as a root of $P_n(z)-zQ_n(z)$ is 1.
\end{proof}
\begin{rmk}\label{dist}
For $F=H$ from Corollary \ref{Hcomp}, letting $P_n(z)/Q_n(z)$ be a reduced expression of $H^n$ with $P_n,Q_n$ polynomials, we know that the leading terms of $P_n$ and $Q_n$ are never equal, since by Remark \ref{leadingcoefficients}, the leading term of $P_n$ is $(-1)^{n}$ and the leading term of $Q_n$ is $27^n$. And because $\deg H^n=3^n$ by Corollary \ref{Hcomp}, $P_n(z)-zQ_n(z)$ is a degree $3^n$ polynomial. Now that we know that all roots of $P_n(z)-zQ_n(z)$ have multiplicity 1 by Proposition \ref{mult1}, we may conclude that the polynomial must have $3^n$ distinct roots in $\C$.
\end{rmk}
\begin{lemma}\label{rec}
    Let $B_n$ be the number of orbits of size $n$ in $\C$ of the function $H$. Then the sequence $\{B_n\}_{n\in\Z_+}$ can be given recursively as $$B_n=\frac{3^n-\sum_{i|n,i<n}iB_i}{n},$$ with $B_1=3$.
\end{lemma}
\begin{proof}
First note that an orbit of size 1 is a fixed point. The three fixed points of $H$ in $\C$ are 1728, $\frac{3456}{7}(-1-3i\sqrt{3})$, and $\frac{3456}{7}(-1+3i\sqrt{3})$. Note that $\infty\in\hC$ is also a fixed point of $H$, but $\infty\notin\C$, so $B_1=3$.

    Now let us consider the number of fixed points of the function $H^n(z)$ for a given $n\in\N$. Let $\phi$ be a fixed point of $H^n(z)$, so $P_n(\phi)/Q_n(\phi)=\phi$. By Remark \ref{dist}, this polynomial has $3^n$ distinct roots, each with multiplicity 1. So there are $3^n$ fixed points of $H^n(z)$.

    Of these fixed points, we are only interested in those ones $\phi$ such that $H^i(\phi)\neq \phi$ for any $i<n$. Let $\alpha$ be a fixed point of $H^n$ and $H^i$, and not of $H^\ell$ for $\ell<i$ (in other words, $\alpha$ is a point of order $i$). Then $\alpha$ is in an orbit of size $i$, of which there are $B_i$ by definition. Then $\alpha$ is one of the $iB_i$ points of order $i$. If $i<n$, then $i|n$. So we remove from the count of $3^n$ fixed points of $H^n$ the $iB_i$ points of order $i$ for each $i<n$, we get $\ds3^n-\sum_{i|n,i<n}iB_i$ points of order $n$. Dividing this by $n$ gives us the desired number of orbits of size $n$.
\end{proof}
\begin{defn}
    The \textbf{M\"obius function} $\mu:\N\to\{-1,0,1\}$ is defined as $$\mu(n)=\begin{cases}
    -1&\text{$n$ is square-free and has an odd number of prime factors,}\\
    1&\text{$n$ is square-free and has an even number of prime factors,}\\
    0&\text{$n$ is not square-free.}
\end{cases}$$
\end{defn}
    \begin{lemma}\label{mif}
    Let $f$ and $g$ be arithmetical functions and let $n\in\N$. Then $$f(n)=\sum_{d|n}g(d)\text{ if and only if } g(n)=\sum_{d|n}f(d)\mu(n/d).$$ This is known as the \textbf{M\"obius inversion formula}. 
    \end{lemma}
    \begin{proof}
    This is Theorem 2.9 in \cite{A}.
    \end{proof}
\begin{thm}\label{thm2}
    Let $B_n$ be the number of orbits of size $n$ of the function $H$. Then the sequence $\{B_n\}_{n\in\N}$ can be given in closed form as $$B_n=\frac{\sum_{d|n}\mu(d)3^{n/d}}{n}.$$
    \end{thm}
    \begin{rmk}
    This corresponds to sequence \href{https://oeis.org/A027376}{\emph{A027376}} in the \emph{OEIS}. Indeed, this is similar to computation of the number of Hesse loops of length of even length provided in Theorem 22 of \cite{DHH}, which corresponds to the same sequence. Our argument applies more generally to PCF functions of degree 3 and difference-degree 1.
\end{rmk}
\begin{proof}
Recall that $$B_n=\frac{3^n-\sum_{i|n,i<n}iB_i}{n},$$ and so $$3^n=\sum_{i|n}iB_i.$$ Then a straightforward application of the M\"obius inversion formula says that $$nB_n=\sum_{d|n}3^d\mu(n/d)$$ and so $$B_n=\frac{\sum_{d|n}3^d\mu(n/d)}{n}=\frac{\sum_{d|n}\mu(d)3^{n/d}}{n}.$$


\end{proof}
\section{Periodicity under \sectionh}
\subsection{The Hesse pencil}
\begin{defn}
Let $E=V(f)$ be an elliptic curve. The \textbf{Hesse pencil} induced by $E$ is defined as the family of curves $V(tf+u\mathcal{H}(f))$ for $(t,u)\in\P^1$.
\end{defn}
Consider the Hesse pencil $\Pcal$ induced by the curve $E_{0}=V(x^3+y^3+z^3)$. Then $\Pcal$ comprises cubic curves of the form $E_{t}=V(x^3+y^3+z^3-3txyz)$ for $t\in\C$, in addition to $E_\infty=V(xyz)$. We will henceforth refer to $\Pcal$ as the \textbf{standard Hesse pencil}. By using the functions \begin{align*}
a(t)&=-27t(t^3+8),\\
b(t)&=54(t^6-20t^3-8),
\end{align*}
it is known that $E_t\cong V(y^2z-x^3-a(t)xz^2-b(t)z^3)$ \cite{AD}, and thereby we can calculate the $j$-invariant of $E_t$ as \begin{align}\label{form}j(E_t)=1728\frac{4a(t)^3}{4a(t)^3+27b(t)^2}.\end{align}
It is also known that $\haw E_t=E_s$ where $s=\ds\frac{4-t^3}{3t^2}$ \mbox{\cite{CO}}.


We will begin to investigate the behaviors of elliptic curves with periodic $j$-invariants under the iterated Hesse derivative. We will start by looking at specific examples in the standard Hesse pencil.
The following four examples will investigate the behavior of elliptic curves with $j$-invariants of orders 1 and 2 under $H$. These are the same nine $j$-invariants listed in Example \ref{order2}, without the $j=\infty$ case, in which $E$ is not an elliptic curve and $\haw E=E$.
\begin{ex}\label{realfixed} Consider the standard Hesse pencil $\Pcal$. This pencil has six fibers with $j$-invariant equal to $1728$: those with $t$ satisfying the polynomial equation $$t^6-20t^3-8=0.$$ Let $E$ be any of these six fibers. Then $\haw^2 E=E$.
\end{ex}

The previous is an example of what is called a \textbf{harmonic} cubic, which are well documented in for example Proposition 4.13 of \cite{CO}. The following three examples can be verified using a computer algebra system such as Macaulay2.
\begin{ex}\label{nonrealfixed}
Consider the standard Hesse pencil $\Pcal$. Let $E\in\Pcal$ have a $j$-invariant $j(E)=\ds\frac{3456}{7}\left(-1\pm3i\sqrt{3}\right)$. Then $j(E)$ is fixed under $H$, and $\haw^3E=E$.

Note here that for $j(E_t)=\ds\frac{3456}{7}\left(-1\pm3i\sqrt{3}\right)$, $t$ must solve the equation in Line (\mbox{\ref{form}}) $$\frac{6912a(t)^3}{4a(t)^3+27b(t)^2}=\frac{3456}{7}\left(-1\pm3i\sqrt{3}\right).$$ Simplifying this equation ultimately yields the polynomial $$a(t)^6+9a(t)^3b(t)^2+27b(t)^4=0.$$ By defining $\normalfont{\texttt{t}}$ to satisfy this polynomial in a field extension or quotient ring in Macaulay2 \cite{M2}, we can then compute $\mathcal{H}^3(x^3+y^3+z^3-3t xyz)$ as $\normalfont{\texttt{h3}}$ and observe that \begin{center}\normalfont{\texttt{coefficient(x\^{}3,h3)*(-3*t)==coefficient(x*y*z,h3)}}\end{center} yields \normalfont{\texttt{true}}, as required. 
\end{ex}
The remaining examples can be verified similarly.
\begin{ex}
Consider the standard Hesse pencil $\Pcal$. Let $E\in\Pcal$ have a $j$-invariant of $3456\left(5\pm3\sqrt{3}\right)$. Then $j(E)$ is periodic under $H$ of order $2$, and $\haw^4E=E$, and $\haw^2E\neq E$.
\end{ex}
\begin{ex}
Consider the standard Hesse pencil $\Pcal$. Let $E\in\Pcal$ have a $j$-invariant satisfying the polynomial equation $$13j(E)^4-89856j(E)^3+4586471424j(E)^2-5283615080448j(E)+2282521714753536=0.$$ Then $j(E)$ is periodic under $H$ of order $2$, and $\haw^6E=E$, and $\haw^2 E\neq E$. Note that the normalized defining polynomial for the number field defined by the above polynomial is $x^4 - x^3 - x^2 + x + 1$ \cite{lmfdb}.
\end{ex}


\subsection{The dynamics of the Hesse derivative}
We will now investigate the behavior of elliptic curves with periodic $j$-invariants in general. We will require two lemmas and the following definition.
\begin{defn}
    Let $C=V(f)\subseq\P^2$ be a curve of degree $d$ and let $Q=(Q_0,Q_1,Q_2)\in\P^2$ be any point. Then the \textbf{polar} (or \textbf{first polar}) of $C$ with respect to $Q$ is $$\mathfrak{P}_Q(C)=V(Q_0\cdot f_x+Q_1\cdot f_y+Q_2\cdot f_z).$$ It is the unique degree $d-1$ curve whose intersection with $C$ (counting with multiplicity) is exactly the points of $C$ whose tangent lines contain $Q$. Furthermore, for $C$ of degree at least $3$, the \textbf{polar conic} of $C$ with respect to $Q$, denoted $\mathfrak{P}^{(2)}_Q(C)$, is the result of iterating the polar curve with respect to $Q$ until one obtains a degree $2$ curve.
\end{defn}
 
\begin{lemma}
    Given a curve $C\subseq\P^2$ of degree $d$ and any $Q\in\P^2$ and any $A\in\emph{Aut}(\P^2)$, one has $$A(\mathfrak{P}_Q(C))=\mathfrak{P}_{AQ}(AC).$$
\end{lemma}
\begin{proof}
Note again that $\Pk_Q(C)$ is the unique degree $d-1$ curve whose intersection with $C$ is exactly the points of $C$ whose tangent lines contain $Q$, counting with multiplicity. In other words, $$\Pk_Q(C).C=P_1+P_2+\cdots+P_{d(d-1)}$$ where $$\{Q\}=\bigcap_{i=1}^{d(d-1)}T_{P_i}(C).$$ Note that in some cases, like if $Q\in C$, we can have $P_i=P_j$ for $i\neq j$.

Since automorphisms preserve incidence and linearity, we have $$A(\Pk_Q(C)).AC=AP_1+AP_2+\cdots+AP_{d(d-1)}$$ where $$\{AQ\}=\bigcap_{i=1}^{d(d-1)}T_{AP_i}(AC).$$ Thus $A(\Pk_Q(C))$ is a degree $d-1$ curve whose intersection with $AC$ is exactly the points whose tangent lines contain $AQ$. Since $\Pk_{AQ}(AC)$ is the unique curve with such properties, we must have $A(\Pk_Q(C))=\Pk_{AQ}(AC)$.
\end{proof}
\begin{lemma}\label{polarcommute}
    Given a curve $C\subseq\P^2$ and an automorphism $A\in\emph{Aut}(\P^2)$, we have $\haw AC=A\haw C$.
\end{lemma}
\begin{proof}
    Note that $\haw C$ satisfies $$\haw C=\{Q\in\P^2:\Pk^{(2)}_Q(C)\text{ is reducible}\},$$ see Section 1.1.4 of \cite{D}.
    By the previous lemma, $A(\Pk_Q(C))=\Pk_{AQ}(AC)$, so automorphisms commute with finding the polar curve. Thus we have $A(\Pk^{(2)}_Q(C))=\Pk^{(2)}_{AQ}(AC)$. So if $Q\in\haw C$, then $\Pk^{(2)}_Q(C)$ is reducible and so $\Pk^{(2)}_{AQ}(AC)=A(\Pk^{(2)}_Q(C))$ is reducible. Thus $AQ\in\haw AC$, so $A\haw C\subseq \haw AC$. A similar argument works for the reverse containment. Thus we have equality.
\end{proof}

We will also need to understand the automorphism group on the standard Hesse pencil $\Pcal=V(x^3+y^3+z^3-3txyz)$, known as the \textbf{Hessian group} and denoted $G_{216}$ \cite{AD}.

Consider the homomorphism $G_{216}\to\Aut(\P^1)$. The image of this homomorphism is isomorphic to $A_4$, which induces the tetrahedral action on $\P^1$, yielding two orbits of size 4, one orbit of size 6, and every other point of $\P^1$ in an orbit of size 12, see \cite{AD}. Furthermore, we can write $$G_{216}=\Gamma\rtimes\Delta,$$ where $\Gamma\cong(\Z/3\Z)^2$ preserves each fiber of $\Pcal$ (that is, $g(E_t)=E_t$ for all $t\in\P^1$ and for all $g\in\Gamma$), and $\Delta\cong\text{SL}_2(\F_3)$ acts on $\Gamma$ by the natural linear representation. The subgroup $\Delta$ can be generated as \begin{align}\label{Delta}\Delta=\left\langle\alpha=\begin{pmatrix}1&1&1\\1&\zeta_3&\zeta_3^2\\1&\zeta_3^2&\zeta_3\\\end{pmatrix},\beta=\begin{pmatrix}1&0&0\\0&\zeta_3&0\\0&0&\zeta_3\\\end{pmatrix}\right\rangle\leq\Aut(\P^2)=\text{PGL}(3),\end{align} where $\zeta_3$ is a primitive third root of unity \mbox{\cite{AD}}. Note that
 $\alpha^2=\begin{pmatrix}1&0&0\\0&0&1\\0&1&0\\\end{pmatrix}\in\text{Stab}(t)$ for all $t\in\P^1$. Note that $\Delta/\langle\alpha^2\rangle\cong A_4$ and $\Gamma\rtimes\langle\alpha^2\rangle=\ker(G_{216}\to\Aut(\P^1))$.

Since up to change of projective coordinates any Hesse pencil $\Pcal(E)$ can be brought to the standard Hesse pencil, we have $\Aut(\Pcal(E))\cong(\Z/3\Z)^2\rtimes\text{SL}_2(\F_3)$, where $\text{im}(\Aut(\Pcal(E))\to\Aut(\P^1))\cong A_4$ induces the tetrahedral action on $\P^1$, and $\ker(\Aut(\Pcal(E))\to\Aut(\P^1))\cong C_3^2\rtimes C_2$, so there are 12 automorphisms of $\Pcal(E)$ that do not fix all fibers.

For a general elliptic curve $E$, let us denote by $\Gamma(E)$ and $\Delta(E)$ the subgroups of $\Aut(\Pcal(E))$ such that $\Gamma(E)\cong C_3^2$, $\Delta(E)\cong\text{SL}_2(\F_3)$, and $\Aut(\Pcal(E))=\Gamma(E)\rtimes\Delta(E)$.

\begin{thm}\label{ffo}
Let $E$ be an elliptic curve and let $j(E)\in\C$ be a periodic point of $H$ of orbit size $p$. Then there exists an $n\in\N$ such that $\haw^nE=E$. Furthermore, this $n$ is at most $3p$.
\end{thm}
\begin{proof}
Consider the Hesse pencil $\Pcal(E)$ induced by $E$. The pencil $\Pcal(E)$ is an elliptic fibration parametrized by $t\in\P^1$. Let us denote the fibers of $\Pcal(E)$ as $E_t$ and let us denote $E_{t_0}=E$. By Corollary 2.2 of \cite{ATT}, $\haw^nE_{t_0}$ is in $\mathcal{P}(E)$ for every $n\in\N$ (i.e., all elliptic curves in the Hesse pencil share the same 3-torsion points). It is also known (shown, for example, in the proof of Proposition 9.5 of \cite{ATT}) that a Hesse pencil contains at most 12 fibers isomorphic to $E_{t_0}$ (including $E_{t_0}$ itself). Let $O(E_{t_0})$ be the set of the fibers of the Hesse pencil that are isomorphic to $E_{t_0}$, hence $\#O(E_{t_0})\leq 12$.

 Now let $j_{t_0}:=j(E_{t_0})$ be periodic of order $p$ under $H$, so $H^p(j_{t_0})=j_{t_0}$ and $H^{p'}(j_{t_0})\neq j_{t_0}$ for any $p'<p$. Then $\haw^p$ acts on $O(E_{t_0})$. Now we would like to show that the action of $\haw^p$ on $O(E_{t_0})$ is in fact a permutation.
 
 
 Consider the automorphism group $\Aut(\Pcal(E))=\Gamma(E)\rtimes \Delta(E)$, where $\Delta(E)=\langle A,B\rangle\cong\langle\alpha,\beta\rangle$ as in Line (\ref{Delta}), for $A,B\in\Aut(\P^2)$. Specifically, where $\ker(\Aut(\Pcal(E))\to\Aut(\P^1))=\Gamma(E)\rtimes\langle A^2\rangle$.
 
 For the cases that $t_0\in\P^1$ is in an orbit of size 4 or 6 under the image of $\Aut(\Pcal(E))\to\Aut(\P^1)$, we have $j(t_0)\in\{0,\infty,1728\}$, for which the statement of the theorem is easy to verify computationally.
 
 Now let $t_0\in\P^1$ be in an orbit of size $12$ under the image of $\Aut(\Pcal(E))\to\Aut(\P^1)$, and so $\#O(E_{t_0})=12$. Now we can write $O(E_{t_0})=\{E_{t_0},g_1(E_{t_0}),\dots,g_{11}(E_{t_0})\}$ for $g_i\in \Aut(\Pcal(E_{t_0}))$, where $g_i(E_{t_0})\neq g_j(E_{t_0})$ for $i\neq j$. We want to show that the function $\haw^p|_{O(E_{t_0})}$ is injective. 
 
 Suppose $\haw^pg_i(E_{t_0})=\haw^pg_j(E_{t_0})$. Then by Lemma \ref{polarcommute}, $g_i(\haw^pE_{t_0})=g_j(\haw^pE_{t_0})$. Denoting $\haw^p E_{t_0}$ as $E_{t_k}$, we then have $g_i(E_{t_k})=g_j(E_{t_k})$, so $g_ig_j\inv\in\text{Stab}(t_k)$. By the orbit-stabilizer theorem, $\#\text{Stab}(t_k)=18$, so $\text{Stab}(t_k)=\Gamma(E)\rtimes\langle A^2\rangle$. Thus $g_ig_j\inv\in\Gamma(E)\rtimes\langle A^2\rangle=\text{Stab}(t_0)$, and so we have $g_i(E_{t_0})=g_j(E_{t_0})$. Therefore $\haw^p|_{O(E_{t_0})}$ is injective. By the pigeonhole principle, $\haw^p|_{O(E_{t_0})}$ must also be surjective and therefore a permutation on $O(E_{t_0})$. Since permutations on finite sets have finite order, we know that there exists an $n\in\N$ such that $\haw^n|_{O(E_{t_0})}=\Id_{O(E_{t_0})}$, and so $E_{t_0}$ must have a finite $\haw$-orbit.
 
  Note that, by the $\Delta(E)/\langle A^2\rangle\cong A_4$-action on the fibers, the size of the $\haw$-orbit of each elliptic curve is at most three times the $H$-orbit of the $j$-invariant associated to that curve. Suppose $\haw^p E=g(E)$ for some $g\in \text{Aut}(\Pcal(E))$. Then the image of $g$ in the projection to $\Delta(E)/\langle A^2\rangle$ must have order $1$, $2$, or $3$. Since $\haw^p$ commutes with $g$, the order of $E$ under $\haw^p$ must be the same.
 \end{proof}
  

 Theorem \ref{ffo} demonstrates the existence of an $n\in\N$ such that $\haw^nE=E$ for an elliptic curve with periodic $j$-invariant of order $p$, but does not describe how to find the minimal positive such $n$. Indeed, Example \ref{realfixed} and Example \ref{nonrealfixed} both describe the behavior of elliptic curves with fixed $j$-invariants under $H$, but in one example the order of $E$ under $\haw$ is 2 and in the other the order is 3. This leads us to the following question:
 
 \begin{qstn}
 Let $E$ be an elliptic curve and let $j(E)$ be a point of order $p$ under $H$. How may we compute the smallest positive $n$ such that $\haw^nE=E$?
 \end{qstn}
 Following the same technique as that outlined in Example \ref{nonrealfixed}, we have collected the following tables of data on elliptic curves with periodic $j$-invariants. Specifically, we construct an elliptic curve of the form $x^3+y^3+z^3-3txyz$ given a minimal polynomial of the $j$-invariant by using the formulae $a(t)$ and $b(t)$ to construct a minimal polynomial for $t$. We note that since two isomorphic elliptic curves are equal up to change of $\P^2_\C$-coordinates (follows, e.g., from Proposition III.3.1 of \cite{S}) and $\haw$ commutes with automorphisms of the plane, the specific choice of $E$ for each $j$-invariant does not matter. Table \ref{orb123} is organized by the minimal polynomial $m(X)$ over $\Q$ of $j(E)/6912$ instead of $j(E)$, as this quantity yields much more manageably-sized coefficients to present in the table. In Table \ref{orb456}, we have substituted the minimal polynomials for their degrees to save space. The minimal polynomials are found as factors of the numerator of $H^n(X)-X$ for $1\leq n\leq 6$. For some particularly notable $j$-invariants-- that is, those for which $\#\text{Orb}_H(j(E))=\#\text{Orb}_{\haw}(E)$-- the minimal polynomial of $j(E)/6912$ is referenced and listed beneath Table \ref{orb456}. 


 \begin{table}[H]
 \[\begin{array}{|c|c|c|}
 \hline
 m_{j(E)/6912}(X)&\#\text{Orb}_H(j(E))&\#\text{Orb}_{\footnoteh}(E)\\
 \hline
 \hline
 4X-1&1&2\\
 \hline
 7X^2+X+1&1&3\\
 \hline
 2X^2-10X-1&2&4\\
 \hline
 13X^4-13X^3+96X^2-16X+1&2&6\\
 \hline
 \makecell{X^6+21X^5+690X^4\\-749X^3+798X^2-33X+1}&3&6\\
 \hline
 \makecell{19X^6+228X^5+1140X^4\\-848X^3+186X^2+3X+1}&3&9\\
 \hline
 \makecell{37X^{12}+3885X^{11}+126651X^{10}\\-477337X^9+337995X^8+641025X^7\\+92400X^6-
      280215X^5+102501X^4\\-17023X^3+1515X^2+6X+1}&3&9\\
      \hline
      \end{array}
      \]
       \caption{Orbit Sizes under $H$ from 1 to 3}
    \label{orb123}
      \end{table}
      
            \begin{table}[H]
      \[\begin{array}{|c|c|c|}
 \hline
 \deg m_{j(E)}(X)&\#\text{Orb}_H(j(E))&\#\text{Orb}_{\footnoteh}(E)\\
 \hline
 \hline
 4\text{ (line (\ref{m4}))}&4&4\\
 \hline
 8&4&8\\
 \hline
 12&4&8\\
 \hline
 24&4&12\\
 \hline
 24&4&12\\
 \hline
 10&5&15\\
 \hline
 20\text{ (line (\ref{m5}))}&5&5\\
      \hline
       60&5&10\\
 \hline
 60&5&15\\
 \hline
  90&5&15\\
 \hline
 \end{array}\hspace{0.5cm}
 \begin{array}{|c|c|c|}
 \hline
 \deg m_{j(E)}(X)&\#\text{Orb}_H(j(E))&\#\text{Orb}_{\footnoteh}(E)\\
 \hline
 \hline
  6\text{ (line (\ref{m6}))}&6&6\\
 \hline
 12&6&12\\
 \hline
 18&6&6\\
  \hline
 24&6&6\\
 \hline
 24&6&12\\
 \hline
 72&6&12\\
 \hline
 72&6&12\\
      \hline
       216&6&18\\
 \hline
 252&6&18\\
 \hline
      \end{array}
      \]
             \caption{Orbit Sizes under $H$ from 4 to 6}
    \label{orb456}
      \end{table}
      \begin{align}\label{m4}
      m_{j(E)/6912}(X)=5X^4-95X^3-15X^2+25X-1
      \end{align}
      \begin{align}\label{m5}
      m_{j(E)/6912}(X)=61X^{20}+30439X^{19}+4239622X^{18}-62797182X^{17}+21844344X^{16}\\+2420526909X^{15}+3841771704X^{14}-12551895336X^{13} \nonumber\\+17727114291X^{12}-13786857050X^{11}+7242436303X^{10} \nonumber\\-1043771609X^{9}-427615959X^{8}+75464286X^{7} \nonumber\\+38297436X^{6}-13780329X^{5}+1911756X^{4} \nonumber\\-142377X^{3}+7081X^{2}+10X+1 \nonumber
      \end{align}
      \begin{align}\label{m6}
     m_{j(E)/6912}(X)=X^6-141X^5-363X^4+1924X^3-741X^2+48X+1
      \end{align}
      
%
 
 
 The most intriguing revelation of these data is the existence of elliptic curves $E$ whose orbit size under $\haw$ is equal to the orbit size of $j(E)$ under $H$; this is a property that is not exhibited in Table \ref{orb123}. Furthermore, the size of $\text{Orb}_{\footnoteh}(E)$ is not necessarily positively correlated with the degree of $m_{j(E)}$, even among $j$-invariants of the same orbit size under $H$. For example, note the $j$-invariants of degrees $10$ and $20$ in Table \ref{orb456}. In the future it may be interesting to replicate these tables over a field of positive characteristic, where the large coefficients will not impose the same limitations on our machines as in characteristic 0.
 \section{Further Work and Questions}
 There is a known connection between the $j$-invariant of an elliptic curve $E$ and its endomorphism ring $\text{End}(E)$. For example, it is shown in Theorem 4.3 of Chapter II of \cite{Sad} that for an elliptic curve $E$ with $j$-invariant $j(E)\in\overline{\Q}$, $[\Q(j(E)):\Q]=h_K$, where $h_K$ is the class number of the field $K=\text{frac}(\text{End}(E))$. This brings us to the following question.
 
 \begin{qstn}
 Can we classify all elliptic curves $E$ for which $\emph{End}(E)=\emph{End}(\haw E)$?
 \end{qstn}
 
 It is worth noting that an elliptic curve $E$ is not necessarily isogenous to its Hesse derivative over $\Q$, even when $j(E)$ has finite orbit over $H$. For example, the curve $E=V(y^2z-x^3-9xz^2)$ has $j$-invariant $1728$ and has Hesse derivative $\haw E=V(24xy^2+216x^2z-648z^3)=V(xy^2+9x^2z-27z^3)$, which is $\Q$-isomorphic to $V(y^2z-x^3+3xz^2)$ under the map $x\mapsto z$, $y\mapsto y$, $z\mapsto x/3$. We can see then that $E$ has isogeny class \href{https://www.lmfdb.org/EllipticCurve/Q/576/c/4}{\textsf{576.c}} and $\haw E$ has isogeny class \href{https://www.lmfdb.org/EllipticCurve/Q/576/i/1}{\textsf{576.i}} in the L-functions and modular forms database \cite{lmfdb}. We do however have $\text{End}(E)=\text{End}(\haw E)=\Z$, and $\text{End}(E_{\overline{\Q}})=\text{End}(\haw E_{\overline{\Q}})=\Z[\sqrt{-1}]$.
 

It is also interesting to further investigate the dynamics of the function $H$. The following question comes from a collaboration between the author and Chris Peterson.
\begin{qstn}
    For which $j\in\Q$ and $q$ prime does the sequence $H^n(j)$ converge in the $q$-adic metric?
\end{qstn}
Experimental evidence computed in Maple appears to corroborate the following conjecture, which comes from a collaboration between the author and Chris Peterson, but their proofs (if they exist) remain elusive.
\begin{conj}
    Let $j\in\R$ be generic and consider the sequence $\{H^n(j)\}_{n\in\N}$. Define \begin{align*}L_n&=\#\{H^i(j):i\leq n,H^i(j)< 0\},\\
    M_n&=\#\{H^i(j):i\leq n,0<H^i(j)< 6912\},\\
    R_n&=\#\{H^i(j):i\leq n,6912<H^i(j)\}.\end{align*}
    Then $$\lim_{n\to\infty}\frac{L_n}{n}=\lim_{n\to\infty}\frac{M_n}{n}=\lim_{n\to\infty}\frac{R_n}{n}=\frac{1}{3}.$$
    \end{conj}

\end{document}